\documentclass{amsart} 
\address{Department of Mathematics, Institute of Science Tokyo, 2-12-1 Ookayama, Meguro-ku, Tokyo, 152-8551, Japan}
\email{tanaka.a.2255@m.isct.ac.jp}

\usepackage{amsmath,amssymb}
\usepackage{amsfonts}
\usepackage{graphicx} 

\usepackage{hyperref}
\usepackage{caption} 
\usepackage{mathtools}
\usepackage{enumitem}
\usepackage[margin=10pt]{caption} 

\usepackage[final]{pdfpages}
\usepackage{amsthm}
\usepackage{color}

\usepackage{tikz}  

\theoremstyle{plain}
\newtheorem{thm}{Theorem}[section]

\newtheorem{lem}[thm]{Lemma}

\newtheorem*{thm*}{Theorem}
\newtheorem*{cor*}{Corollary}

\theoremstyle{definition}
\newtheorem{dfn}[thm]{Definition}
\newtheorem{rem}[thm]{Remark}

\newtheorem*{que*}{Question}
\newtheorem*{con*}{Conjecture}

\begin{document}

\title[A Simple Construction of Lefschetz Fibrations]{A Simple Construction of Lefschetz Fibrations on Compact Stein Surfaces}
\author{Atsushi Tanaka}
\date{}

\begin{abstract}

Loi--Piergallini, Akbulut--Ozbagci, and Akbulut--Arikan showed that every compact Stein surface admits a positive allowable Lefschetz fibration over the disk $D^2$ with bounded fibers (PALF in short), and they provided constructions of PALFs corresponding to compact Stein surfaces.

In this paper, we present a simple method for constructing a PALF from a 2-handlebody decomposition of any given compact Stein surface. Our method yields PALFs whose regular fibers have small genus, and it provides an alternative constructive proof of the above result.

We also define the minimal genus of a regular fiber of a PALF on the knot trace of a knot $K$ with framing one less than its maximal Thurston--Bennequin number as an invariant of $K$. When the grid number of $K$ is $N$, our construction produces a PALF whose regular fiber has genus at most $(N - 1)/2$, showing that the defined invariant is bounded above by $(N - 1)/2$.

\end{abstract}

\maketitle

\section{Introduction}\label{sec:intro}

A smooth 4-manifold admitting a handle decomposition consisting of a 0-handle, 1-handles, and 2-handles is called a 4-dimensional 2-handlebody. Eliashberg \cite{MR1044658} and Gompf \cite{MR1668563} showed that a smooth 4-dimensional 2-handlebody admits the structure of a compact Stein surface (a Stein surface in short) if each 2-handle is attached along a Legendrian knot in the boundary of the 1-handlebody, with a framing one less than its Thurston--Bennequin number. Conversely, any Stein surface admits such a handle decomposition.

Among Lefschetz fibrations, those having the 2-disk as their base space and a compact surface with boundary as their regular fiber are called positive allowable Lefschetz fibrations (PALFs in short). Harer \cite{MR2628695} (see also Etnyre and Fuller \cite{MR2219214}) presented a construction method for achiral allowable Lefschetz fibrations (AALFs in short), a broader class containing PALFs: if $X$ is a 4-dimensional handlebody consisting of handles of index at most two, then $X$ admits an AALF. Harer's method for constructing AALFs serves as a key reference for our approach.

Loi and Piergallini \cite{MR1835390} proved the equivalence between Stein surfaces and PALFs. Subsequently, Akbulut and Ozbagci \cite{MR1825664} and Akbulut and Arikan \cite{MR2972525} gave constructive proofs of this equivalence by explicitly constructing PALFs corresponding to given Stein surfaces. To construct a PALF diffeomorphic to a Stein surface given by a handle decomposition, the method of Akbulut and Ozbagci \cite{MR1825664} uses the Seifert surface of the torus knot $T_{N,N+1}$ as the regular fiber of the PALF, assuming the attaching circle of the 2-handles has grid number $N$. That is, the regular fiber is constructed by plumbing on the order of $N^2$ Hopf bands. Consequently, the genus of the regular fiber is roughly on the order of $(N-1)^2/2$. Later, Akbulut and Arikan \cite{MR2972525} succeeded in reducing the genus of this regular fiber. The equivalence between Stein surfaces and PALFs, along with these constructive methods, has profoundly influenced developments in this area.

Ukida \cite{MR3571042, MR4405983, Ukida} developed a method for constructing PALFs via Kirby calculus and applied it to produce genus-zero PALFs for the Akbulut cork and the Akbulut--Yasui plug. A manuscript in preparation by Ukida \cite{Ukida}, which introduces a method to construct small-genus PALFs from handle decompositions of Stein surfaces using Kirby diagrams and moves, has been a significant source of inspiration for our present work.

This paper provides a simple alternative construction of a PALF corresponding to a Stein surface, differing from the approaches of Akbulut and Ozbagci \cite{MR1825664} and Akbulut and Arikan \cite{MR2972525}. As a consequence, relying solely on Kirby calculus, we obtain an alternative constructive proof that every Stein surface admits a PALF. We also define the minimal genus of a regular fiber of a PALF on the knot (or link) trace of $K$ with a framing one less than its maximal Thurston--Bennequin number, as a new invariant of the knot (or link) $K$. When the grid number of $K$ is $N$, the construction presented in this paper yields a PALF by plumbing at most $N-1$ Hopf bands. Consequently, the regular fiber of this PALF has genus at most $(N-1)/2$, and hence our defined invariant is bounded above by $(N-1)/2$.

This paper is organized as follows. Section~\ref{sec:preliminaries} reviews the necessary definitions, basic properties, and background theorems. In Section~\ref{sec:construction}, we state and prove our main theorem, which provides a method for constructing a PALF from a 2-handlebody decomposition of any given Stein surface. Finally, in Section~\ref{sec:genus}, we formally define the minimal genus of a regular fiber of a PALF on the trace of a knot (or link) $K$, equipped with a framing one less than its maximal Thurston--Bennequin number, as a new invariant of $K$. We then prove that if the grid number of $K$ is $N$, this invariant is bounded above by $(N-1)/2$.

\section*{Acknowledgement}
The author would like to express his sincere gratitude to Professor Hisaaki Endo for his invaluable guidance and continuous support. This paper addresses one of the research topics that Professor Endo has consistently pursued with his students in the laboratory. The author is also grateful to Takuya Ukida for showing him a manuscript of the paper \cite{Ukida} in preparation, and to his colleagues at the Institute of Science Tokyo for many helpful discussions and insightful comments. This work was supported by JST SPRING, Japan Grant Number JPMJSP2180.

\section{Preliminaries}\label{sec:preliminaries}

In this section, we review the fundamental definitions and theorems used in this paper, following \cite{MR4327688}, \cite{MR1707327}, and \cite{MR2114165}. 
Throughout this paper, we assume all Stein surfaces to be of complex dimension 2, all Lefschetz fibrations to be 4-dimensional, and all contact manifolds to be 3-dimensional.

\subsection{2-handlebody decompositions of Stein surfaces}\label{sec:preliminaries1}

A plane field $\xi$ on a 3-manifold $Y$ is called a contact structure if there exists a 1-form $\alpha$ on $Y$ such that $\xi = \ker(\alpha)$ and $\alpha \wedge d\alpha \neq 0$ at every point of $Y$. In this case, $\alpha$ is called a contact form, and the pair $(Y, \xi)$ is called a contact manifold. 
A Legendrian knot $L$ in $(Y, \xi)$ is a knot in $Y$ that is everywhere tangent to $\xi$. 
A Legendrian isotopy is a smooth isotopy of knots such that the knot is Legendrian at each stage of the isotopy.

Let $(x, y, z)$ denote the standard coordinates on $\mathbb{R}^3$. The 1-form $\alpha_{st} = dz + x\,dy$ is a contact form on $\mathbb{R}^3$, and the associated plane field $\xi_{st} = \ker(\alpha_{st})$ is called the standard contact structure.

For a Legendrian knot $L$ parametrized by $\gamma(t) = (x(t), y(t), z(t))$ in $(\mathbb{R}^3, \xi_{st})$, its projection onto the $yz$-plane is called the front projection. In a front projection, the $x$-coordinate is recovered from the slope of the projected curve via $x = -dz/dy$. Consequently, at any crossing, the strand with the larger slope lies under the strand with the smaller slope. The front projection also features cusp singularities, which correspond to points where the tangent vector to $L$ is parallel to the $x$-axis. Figure~\ref{fig:W0201} illustrates a front projection of a Legendrian representative of the right-handed trefoil knot $T_{2,3}$.

\begin{figure}[htbp]
\centering
\includegraphics[width=8cm]{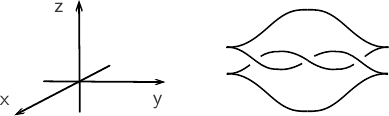} 
\caption{A front projection of a Legendrian right-handed trefoil knot.}
\label{fig:W0201}
\end{figure}

The Thurston--Bennequin number of a Legendrian knot $L$ in $(\mathbb{R}^3, \xi_{st})$ is invariant under Legendrian isotopy and can be computed from a front projection using the formula
\[
tb(L) = wr(L) - \lambda(L),
\]
where $wr(L)$ is the writhe of $L$ and $\lambda(L)$ is the number of left cusps in the projection. The maximal Thurston--Bennequin number $\overline{tb}(K)$ of a topological knot type $K$ is the maximum value of the Thurston--Bennequin number attained by any Legendrian representative of $K$.

A complex manifold is called a Stein manifold if it admits a proper holomorphic embedding into $\mathbb{C}^n$ for some $n$. A Stein surface is defined as a sublevel set of an exhausting strictly plurisubharmonic function on a Stein manifold of complex dimension 2. Stein surfaces became accessible to topological study after they were characterized in terms of handlebody theory by Eliashberg and Gompf.

\begin{thm}[Eliashberg \cite{MR1044658}, Gompf \cite{MR1668563}]\label{thm:EliashbergGompf}
A smooth 4-manifold admitting a handle decomposition consisting of a 0-handle, 1-handles, and 2-handles admits the structure of a Stein surface if the 2-handles are attached to the Stein surface $\natural^\ell (S^1 \times D^3)$ along Legendrian knots in the standard contact manifold $\sharp^\ell (S^1 \times S^2)$, with framings one less than their respective Thurston--Bennequin numbers. Conversely, any Stein surface admits such a handle decomposition.
\end{thm}

Figure~\ref{fig:W0202} shows an example of a Legendrian link diagram in standard form, illustrating the handle decomposition described in Theorem~\ref{thm:EliashbergGompf}.

\begin{figure}[htbp]
\centering
\includegraphics[width=8cm]{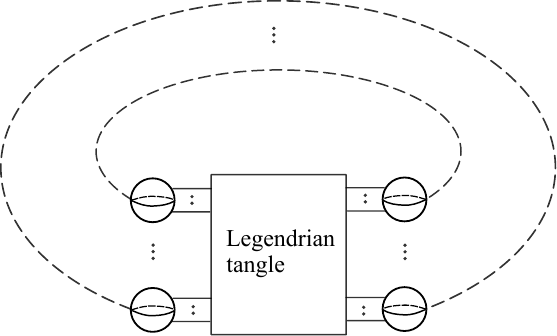}
\caption{A Legendrian link diagram in standard form.}
\label{fig:W0202}
\end{figure}

\subsection{Representation of Legendrian knots using grid diagrams}\label{sec:preliminaries2}

In this subsection, we review the representation of Legendrian knots and links using grid diagrams (cf.\ \cite[Chapter 3]{MR3381987}, \cite{MR2500576}). A grid diagram with grid number $N$ is an $N \times N$ square grid with $N$ $X$'s and $N$ $O$'s placed in distinct squares, such that each row and each column contains exactly one $X$ and one $O$. Any such grid diagram defines an oriented knot or link diagram according to the following standard convention: for each row, draw a horizontal segment from the $O$ to the $X$, and for each column, draw a vertical segment from the $X$ to the $O$. The vertical segments always cross over the horizontal ones.

An example of a grid diagram $G$ is shown in Figure~\ref{fig:W0203}(a). In this paper, the orientation of knots and links represented by grid diagrams is irrelevant. Thus, the $X$ and $O$ markings are omitted, and the vertical and horizontal segments are depicted simply as in Figure~\ref{fig:W0203}(b). We refer to such a diagram as a knot (or link) in grid position. We index the columns from left to right and the rows from bottom to top; for instance, in the $3 \times 3$ grid diagram of Figure~\ref{fig:W0203}, they are referred to as the first, second, and third columns (resp.\ rows).

\begin{figure}[htbp]
\centering  
\begin{tikzpicture}
    \node[anchor=south west, inner sep=0] (image) at (0,0)  {\includegraphics{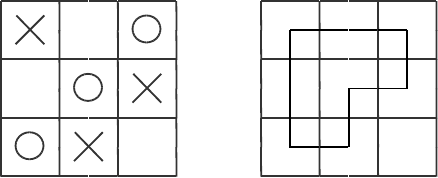}};
    \begin{scope}[x={(image.south east)},y={(image.north west)}]
        \node [below] at (0.2, -0.02) {(a)};
        \node [below] at (0.8, -0.02) {(b)};
    \end{scope}
\end{tikzpicture}
\caption{A grid diagram (a) and a corresponding knot in grid position (b).}
\label{fig:W0203}
\end{figure}

The four types of corners appearing in a knot in grid position are illustrated in Figure~\ref{fig:W0204}.
Figures~\ref{fig:W0204} (a), (b), (c), and (d) show the northwest (NW), northeast (NE), southwest (SW), and southeast (SE) corners, respectively.

\begin{figure}[htbp]
\centering  
\begin{tikzpicture}
    \node[anchor=south west, inner sep=0] (image) at (0,0)  {\includegraphics[width=\textwidth]{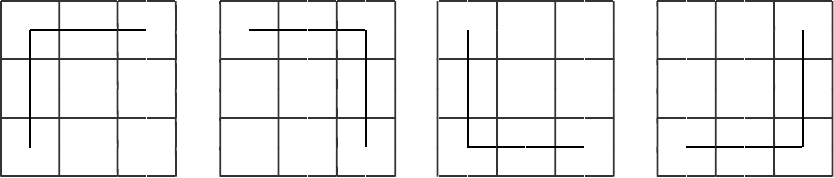} };
    \begin{scope}[x={(image.south east)},y={(image.north west)}]
        \node [below] at (0.11, -0.02) {(a)};
        \node [below] at (0.37, -0.02) {(b)};
        \node [below] at (0.63, -0.02) {(c)};
        \node [below] at (0.89, -0.02) {(d)};
    \end{scope}
\end{tikzpicture}
\caption{The four types of corners appearing in a knot in grid position.}
\label{fig:W0204}
\end{figure}

The procedure for converting a front projection of a Legendrian knot (or link) $\widetilde{C}$ in $(S^3, \xi_{st})$ into a knot in grid position is as follows. First, rotate the front projection of $\widetilde{C}$ by $45^\circ$ clockwise. Then, deform it into a knot $\widetilde{C'}$ in grid position such that the left cusps are mapped to northwest (NW) corners and the right cusps to southeast (SE) corners. Consequently, the number of NW corners of $\widetilde{C'}$ equals the number of left cusps of the original Legendrian knot $\widetilde{C}$.

As an example, consider the Legendrian representative of the right-handed trefoil knot $T_{2,3}$ shown in Figure~\ref{fig:W0205}(a), which has writhe $3$ and two left cusps. If a 2-handle attached along this Legendrian knot is to yield a Stein surface, Theorem~\ref{thm:EliashbergGompf} dictates that the framing must be $0$ (since $tb(\widetilde{C}) - 1 = (3 - 2) - 1 = 0$). Figure~\ref{fig:W0205}(b) shows the corresponding knot $\widetilde{C'}$ in grid position. The two left cusps, highlighted by pale red circles in Figure~\ref{fig:W0205}(a), naturally correspond to the two NW corners marked similarly in Figure~\ref{fig:W0205}(b).

Next, to ensure that all vertical segments lifted over 1-handles have northwest (NW) corners as their upper endpoints, we perform the following two operations, (i) and (ii). It is known that under these operations, the resulting Legendrian knot is Legendrian isotopic to the original one (cf.\ \cite{MR2500576}).

\begin{itemize}
\item[(i)] \textbf{SW stabilization at an NE corner.} The writhe of the knot increases by one under this stabilization. Since the number of NW corners (which correspond to left cusps) also increases by one, the Thurston--Bennequin number remains unchanged. As an example, Figure~\ref{fig:W0206}(a) shows the result of applying SW stabilizations to the NE corners in the third and fourth columns of the knot $\widetilde{C'_0}$ in grid position shown in  Figure~\ref{fig:W0205}(b).

\item[(ii)] \textbf{Horizontal commutation.} The vertical segments of length one introduced by the SW stabilizations do not cross any horizontal segments; thus, they do not need to be lifted over 1-handles. We gather these short segments in the rightmost region of the grid diagram. Continuing with our example, Figure~\ref{fig:W0206}(b) illustrates the result of performing a horizontal commutation between the vertical segments in the fourth and fifth columns of Figure~\ref{fig:W0206}(a).
\end{itemize}

Through the procedure described above, we convert the knot $\widetilde{C'_0}$ in grid position into a new knot $\widetilde{C''_0}$ in grid position. 

In our running example, the knot $\widetilde{C''_0}$ in grid position has a writhe of $5$ and four left cusps. Its Thurston--Bennequin number and framing remain $1$ and $0$, respectively, which are unchanged from the original knot. 

\begin{figure}[htbp]
\centering  
\begin{tikzpicture}
    \node[anchor=south west, inner sep=0] (image) at (0,0)  {\includegraphics[scale=1.1]{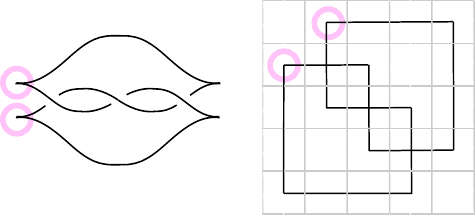} };
    \begin{scope}[x={(image.south east)},y={(image.north west)}]
        \node at (0.12, 0.8) {\small $\widetilde{C_0}$};
        \node at (0.57, 0.5) {\small $\widetilde{C'_0}$};
        \node [below] at (0.25, -0.02) {(a)};
        \node [below] at (0.77, -0.02) {(b)};
    \end{scope}
\end{tikzpicture}
\caption{(a) A front projection of a Legendrian knot. (b) The corresponding knot in grid position.}
\label{fig:W0205}
\end{figure}

\begin{figure}[htbp]
\centering  
\begin{tikzpicture}
    \node[anchor=south west, inner sep=0] (image) at (0,0)  {\includegraphics[scale=0.8]{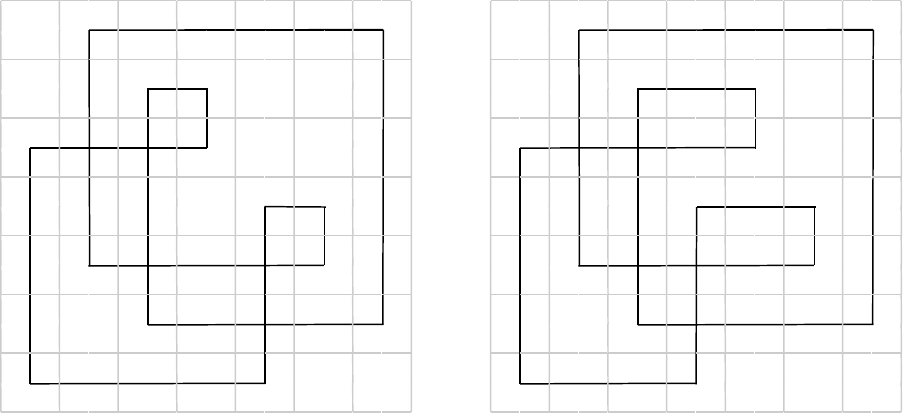}};
    \begin{scope}[x={(image.south east)},y={(image.north west)}]
        \node at (0.55, 0.5) {\small $\widetilde{C''_0}$};
        \node [below] at (0.23, -0.02) {(a)};
        \node [below] at (0.77, -0.02) {(b)};
    \end{scope}
\end{tikzpicture}
\caption{(a) SW stabilization at an NE corner. (b) Horizontal commutation.}
\label{fig:W0206}
\end{figure}

\subsection{Kirby diagrams of PALFs}\label{sec:preliminaries3}

In this subsection, we review the definitions of PALFs and their Kirby diagrams, following \cite[Sections 4.6 and 8.2]{MR1707327}, \cite[Chapter 10]{MR2114165}, and \cite[Sections 2 and 3]{MR4327688}. Let $g \geq 0$ and $b \geq 0$ be integers. We denote by $\Sigma^b_g$ a compact oriented surface of genus $g$ with $b$ boundary components. The mapping class group of $\Sigma^b_g$ is denoted by $\mathrm{MCG}(\Sigma^b_g)$. For a simple closed curve $C$ on $\Sigma^b_g$, we denote by $t_C$ the right-handed Dehn twist along $C$. A PALF is defined as follows.

\begin{dfn}\label{def:PALF}
Let $M$ be a compact, connected, oriented smooth 4-manifold, and let $D^2$ be the standard 2-disk. We refer to $M$ as the total space and $D^2$ as the base space. A surjective smooth map $f \colon M \to D^2$, called the projection, is a \emph{positive allowable Lefschetz fibration (PALF)} if it satisfies the following conditions:
\begin{enumerate}[label=(\roman*)]
\item For every critical point $p \in \mathrm{Int}(M)$ of $f$, there exist local complex coordinates $(z_1, z_2)$ around $p$ and $w$ around $f(p)$ such that $f$ is locally given by $w = f(z_1, z_2) = z_1 z_2$. Here, the orientations determined by $(z_1, z_2)$ and $w$ must be compatible with the given orientations of $M$ and $D^2$, respectively.
\end{enumerate}

By condition (i), each critical point of $f$ is isolated. Since $M$ is compact, the number of critical points of $f$ is finite, and thus the set of critical values $\Delta \subset D^2$ is also finite. The restriction $f|_{f^{-1}(D^2 \setminus \Delta)} \colon f^{-1}(D^2 \setminus \Delta) \to D^2 \setminus \Delta$ is a smooth surface bundle. For each $q \in D^2 \setminus \Delta$, the fiber $F_q = f^{-1}(q)$ is a surface equipped with an orientation compatible with those of $M$ and $D^2$. Because $D^2 \setminus \Delta$ is connected, the diffeomorphism type of $F_q$ is independent of the choice of $q$.

For any $q \in D^2$, we call $F_q$ a regular fiber if $q \notin \Delta$, and a singular fiber if $q \in \Delta$. A singular fiber is obtained from a regular fiber by collapsing a simple closed curve, called a vanishing cycle, to a point. The genus of the regular fiber is referred to as the genus of the PALF.

\begin{enumerate}[label=(\roman*), start=2]
\item The boundary of the regular fiber is non-empty.
\item Every vanishing cycle represents a nontrivial first homology class in the regular fiber.
\end{enumerate}

In this paper, we further assume that the following conditions are satisfied:
\begin{enumerate}[label=(\roman*), start=4]
\item Each singular fiber contains exactly one critical point.
\item No singular fiber contains a sphere of self-intersection $-1$ (i.e., the fibration is relatively minimal).
\end{enumerate}
\end{dfn}

Consider a PALF $f \colon M \to D^2$. Let $\Delta = \{q_1, \ldots, q_n\} \subset \mathrm{Int}(D^2)$ be the set of critical values of $f$, and let $q_0$ be a regular value. Let $\gamma_1, \ldots, \gamma_n \colon [0,1] \to D^2$ be paths from $q_0$ to $q_1, \ldots, q_n$, respectively. The ordered collection of paths $(\gamma_1, \ldots, \gamma_n)$ is called a Hurwitz system for $f$ if it satisfies the following three conditions:
\begin{itemize}
\item Each path $\gamma_i$ is simple (i.e., has no self-intersections).
\item The paths are mutually disjoint except at their common initial point $q_0$; that is, $\gamma_i \cap \gamma_j = \{q_0\}$ for all $i \neq j$.
\item The paths $\gamma_1, \ldots, \gamma_n$ emanate from $q_0$ in counterclockwise cyclic order.
\end{itemize}

\begin{thm}\label{thm:Monodromy}
Let $f \colon M \to D^2$ be a PALF with regular fiber $\Sigma^b_g$. Suppose $f$ admits a Hurwitz system $(\gamma_1, \ldots, \gamma_n)$ with corresponding vanishing cycles $C_1, \ldots, C_n$ in the regular fiber $F_{q_0} \cong \Sigma^b_g$. Then the following hold:
\begin{enumerate}[label=(\roman*)]
\item The total space $M$ is diffeomorphic to the 4-manifold obtained by attaching $n$ 2-handles $h^2_1, \ldots, h^2_n$ to $\Sigma^b_g \times D^2$. The attaching circle of $h^2_k$ is given by $C_k \times \{e^{2\pi i k/n}\} \subset \Sigma^b_g \times \partial D^2$.
\item The framing of each 2-handle $h^2_k$ is $-1$ relative to the surface framing induced by the regular fiber.
\end{enumerate}
\end{thm}

The product of right-handed Dehn twists $t_{C_1} t_{C_2} \cdots t_{C_n}$ defines an element in $\mathrm{MCG}(\Sigma^b_g)$, often referred to as the global monodromy of the PALF. We denote this monodromy factorization by the ordered tuple $(C_1, C_2, \ldots, C_n)$. Here, we adopt the convention that the mapping class group acts on the right; thus, the product is read from left to right, which is the reverse of standard functional composition.

A Kirby diagram is a visual representation of a 4-dimensional handlebody via a link diagram in $\mathbb{R}^3$ (or $S^3$). We now describe the standard procedure for drawing a Kirby diagram of a PALF with regular fiber $\Sigma^b_g$.

The 0-handle of the 4-dimensional handlebody is given by $h^0 \cong D^4$, and the $j$-th 1-handle by $h^1_j \cong D^1 \times D^3$ for $j=1, \ldots, 2g+b-1$. To reflect the product structure of the fibration, we decompose these handles as $h^0 \coloneqq \underline{h^0} \times D^2$ and $h^1_j \coloneqq \underline{h^1_j} \times D^2$, where the corresponding handles forming the 2-dimensional regular fiber $\Sigma^b_g$ are the 2-dimensional 0-handle $\underline{h^0} \cong D^2$ and the $j$-th 1-handle $\underline{h^1_j} \cong D^1 \times D^1$.

\begin{figure}[htbp]
\centering
\begin{tikzpicture}
    \node[anchor=south west, inner sep=0] (image) at (0,0) {\includegraphics[scale=0.8]{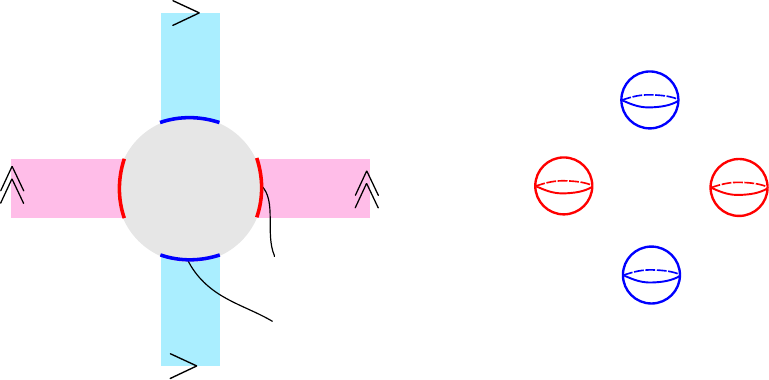}};
    \begin{scope}[x={(image.south east)},y={(image.north west)}]
        \node [below] at (0.25, -0.02) {(a)};
        \node [below] at (0.84, -0.02) {(b)};
        \node at (0.25, 0.5) {\small $\underline{h^0}$};
        \node at (0.25, 0.18) {\small $\underline{h^1_1}$};
        \node at (0.42, 0.5) {\small $\underline{h^1_2}$};
        \node [right] at (0.35, 0.15) {\small attaching region of $\underline{h^1_1}$};
        \node [right] at (0.35, 0.32) {\small attaching region of $\underline{h^1_2}$};
            \end{scope}
\end{tikzpicture}
\caption{(a) The handle decomposition of the regular fiber $T_0 \cong \Sigma^1_1$. (b) A Kirby diagram of $\Sigma^1_1 \times D^2$.}
\label{fig:W0207}
\end{figure}

As an example, we consider $\Sigma^1_1 \times D^2$ and write its handle decomposition as
\[
h \coloneqq \Sigma^1_1 \times D^2 = h^0 \cup h^1_1 \cup h^1_2.
\]
At the center of $h$ lies the 2-dimensional submanifold
\[
\underline{h} \coloneqq T_0 \coloneqq \underline{h^0} \cup \underline{h^1_1} \cup \underline{h^1_2},
\]
which is diffeomorphic to $\Sigma^1_1$. The handle decomposition of the regular fiber $T_0$ is illustrated in Figure~\ref{fig:W0207}(a). The attaching regions of the 1-handles $\underline{h^1_1}$ and $\underline{h^1_2}$ are indicated by blue and red arcs, respectively. The handle decomposition and the attaching regions for the 4-manifold $h = \Sigma^1_1 \times D^2$ are obtained by thickening $\underline{h} = T_0$ in the $D^2$-direction. A Kirby diagram of $\Sigma^1_1 \times D^2$ is shown in Figure~\ref{fig:W0207}(b).

It is worth noting that the structure of the regular fiber of this PALF is equivalent to that obtained by plumbing several Hopf bands to a disk $D^2$. While the specific twisting of the Hopf bands is irrelevant to the diffeomorphism type of the fiber, it shifts the surface framing by $-1$. We will discuss this in Section~\ref{sec:construction}. 
In their work, Akbulut and Ozbagci \cite{MR1825664} constructed PALFs using the Seifert surface of a torus link, which is formed by plumbing Hopf bands. While the construction method presented in this paper also utilizes Hopf bands, it relies entirely on Kirby calculus.

Continuing with our running example, suppose that the PALF admits a Hurwitz system $(\gamma_1, \gamma_2, \gamma_3)$ with corresponding vanishing cycles $C_1, C_2, C_3$, as shown in Figure~\ref{fig:W0208}(a). Consequently, the PALF is described by the monodromy factorization $(C_1, C_2, C_3)$. Let $\overline{C_1}, \overline{C_2}, \overline{C_3}$ denote the attaching circles of the 2-handles corresponding to these vanishing cycles. The resulting Kirby diagram is illustrated in Figure~\ref{fig:W0208}(b). 

By Theorem~\ref{thm:Monodromy}, the framing of each $\overline{C_i}$ is one less than its surface framing. In this drawing convention, the surface framing of each $\overline{C_i}$ coincides with the blackboard framing, which in turn equals its writhe. Consequently, all of the framings are exactly $-1$. Note that in the Kirby diagram, these attaching circles are stacked from bottom to top in the order $\overline{C_1}, \overline{C_2}, \overline{C_3}$, reflecting the cyclic order of the Hurwitz system.

\begin{figure}[htbp]
\centering  
\begin{tikzpicture}
    \node[anchor=south west, inner sep=0] (image) at (0,0)  {\includegraphics[scale=0.8]{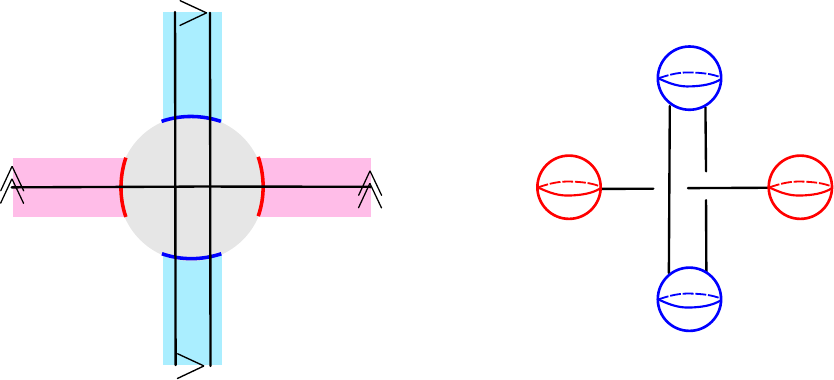} };
    \begin{scope}[x={(image.south east)},y={(image.north west)}]
        \node [below] at (0.23, -0.02) {(a)};
        \node [below] at (0.83, -0.02) {(b)};
        \node at (0.29, 0.8) {\small $C_1$};
        \node at (0.09, 0.46) {\small $C_2$};
        \node at (0.175, 0.8) {\small $C_3$};
        \node at (0.88, 0.675) {\small $\overline{C_1}$};
        \node at (0.88, 0.6) {\small $-1$};
        \node at (0.75, 0.45) {\small $\overline{C_2}$};
        \node at (0.75, 0.54) {\small $-1$};
        \node at (0.78, 0.675) {\small $\overline{C_3}$};
        \node at (0.78, 0.6) {\small $-1$};
    \end{scope}
\end{tikzpicture}
\caption{(a) The handle decomposition of the regular fiber $T_0 \cong \Sigma^1_1$ with vanishing cycles $C_1, C_2, C_3$. (b) A Kirby diagram of the corresponding PALF.}
\label{fig:W0208}
\end{figure}

Summarizing the above discussion, the procedure, denoted by $\Phi$, for constructing a Kirby diagram $KD$ representing a 4-dimensional 2-handlebody diffeomorphic to the total space of a given PALF $P$, using its regular fiber $\Sigma$ and vanishing cycles $C_i$, is formulated as follows:

\begin{itemize}
\item Each vanishing cycle $C_i$ is replaced by the attaching circle $\overline{C_i}$ of a 2-handle. The framing of $\overline{C_i}$ is set to be one less than the surface framing of $C_i$. The vertical ordering (i.e., stacking) of the attaching circles $\overline{C_i}$ from bottom to top is determined by the left-to-right order in the monodromy factorization.

\item The attaching region of each 1-handle in the regular fiber (which is a pair of 1-disks) is replaced by a pair of 3-balls. Alternatively, if one prefers to represent the 1-handles using dotted circles, each 2-dimensional 1-handle is replaced by a corresponding dotted circle.
\end{itemize}

\section{A simple method for constructing a PALF} \label{sec:construction}

In this section, we present a simple method for constructing a PALF on a compact 4-manifold admitting a Stein structure. As discussed in Subsection~\ref{sec:preliminaries1}, Theorem~\ref{thm:EliashbergGompf} \cite{MR1044658, MR1668563} states that such a manifold admits a handle decomposition consisting of a single 0-handle, 1-handles, and 2-handles. These 2-handles are attached along Legendrian knots $\widetilde{C_{01}}, \ldots, \widetilde{C_{0m}}$ in the boundary of the union of the 0- and 1-handles, with a framing of $tb(\widetilde{C_{0k}}) - 1$ for each $1 \leq k \leq m$.

We first describe the construction of a PALF for the simplest case, where the handle decomposition consists only of a 0-handle and a single 2-handle. Subsequently, we generalize this method to the case where the handle decomposition includes a 0-handle, 1-handles, and 2-handles.

\subsection{The case of a 2-handlebody with no 1-handles and one 2-handle}\label{subsec:construction1}

We first describe a method for constructing a PALF $P$ in the case where the original Stein surface $\Pi$ has a handle decomposition consisting of a 0-handle and a single 2-handle attached along a Legendrian knot $\widetilde{C_0}$ in $(S^3, \xi_{st})$.

\begin{itemize}
\item[\textbf{Step 1}] Following the procedure described in Subsection~\ref{sec:preliminaries2}, we convert the given Legendrian knot $\widetilde{C_0}$ in $(S^3, \xi_{st})$ into a knot $\widetilde{C'_0}$ in grid position. Next, to ensure that all vertical segments to be lifted over 1-handles have northwest (NW) corners as their upper endpoints, we apply SW stabilizations at NE corners and subsequent horizontal commutations to any vertical segments whose upper endpoints are northeast (NE) corners. Through these operations, we convert the knot $\widetilde{C'_0}$ into a new knot $\widetilde{C''_0}$ in grid position.
\end{itemize}

We assume that the grid number of the knot $\widetilde{C'_0}$ is $N$.
Let $KD''$ denote the Kirby diagram consisting of a single 2-handle attached to $D^4$ along the attaching circle $\overline{C''_0}$ determined by $\widetilde{C''_0}$. Because the knot $\widetilde{C''_0}$ in grid position represents the same Legendrian knot type as $\widetilde{C_0}$, they yield the same Thurston--Bennequin number. Thus, assigning the framing $tb(\widetilde{C''_0}) - 1$ to the attaching circle $\overline{C''_0}$ ensures that the 4-manifold represented by $KD''$ is diffeomorphic to the Stein surface $\Pi$.

As a concrete example, we first convert the Legendrian knot $\widetilde{C_0}$ into a knot $\widetilde{C'_0}$ in grid position (see Figure~\ref{fig:W0205}). Then, by applying SW stabilizations at NE corners and horizontal commutations to the vertical segments whose upper endpoints are NE corners, we convert $\widetilde{C'_0}$ into the new knot $\widetilde{C''_0}$ in grid position (see Figure~\ref{fig:W0206}). 

Let $KD''$ be the corresponding Kirby diagram with the attaching circle $\overline{C''_0}$ on $D^4$ derived from $\widetilde{C''_0}$ (see Figure~\ref{fig:W0303}(a)).
The 4-manifold represented by $KD''$ is diffeomorphic to the Stein surface $\Pi$. The diagram $\widetilde{C''_0}$ has a writhe of $5$ and four left cusps. Its Thurston--Bennequin number is $1$, and the framing of the attaching circle $\overline{C''_0}$ is $0$.

\begin{itemize}
\item[\textbf{Step 2}] Let $SF(0)$ denote the initial surface consisting of the 0-handle $D^2$, which is represented as a gray square. 
The self-intersections of the closed curve $C_0$ must be eliminated by lifting the vertical segments in columns $2$ through $N-1$ using 1-handles. To guide the subsequent operations, ensuring that the knot $\widetilde{C''_0}$ in grid position can be appropriately placed on the constructed surface at the final stage of the PALF construction, we draw a copy of $\widetilde{C''_0}$ on $SF(0)$. This copy is denoted by $B_0$ and serves as a guide (illustrated by the blue guide line in Figure~\ref{fig:W0301}(a)). 

In our diagrams, the grid lines are omitted for clarity, and column numbers are indicated only where necessary.
\end{itemize}

\begin{itemize}
\item[\textbf{Step 3}] 
We perform a sequence of operations, denoted by $\Theta(j)$ for $1 \leq j \leq N-1$, starting from $SF(0)$. The first operation, $\Theta(1)$, proceeds as follows:
\begin{itemize}
\item In the first column of the grid diagram, there is a vertical segment of the guide line $B_0$. The two endpoints of this segment (marked in purple) lie on the boundary of the 0-handle (see Figure~\ref{fig:W0301}(a)).
\item We then attach a 1-handle at this position (the notation for this 1-handle will be explained shortly). Along with this 1-handle, a new red simple closed curve $C_1$ is introduced. By convention, we arrange each $C_j$ such that its rightmost portion lies near the $N$-th column (for instance, the right portion of $C_1$ is placed in column 6 in Figure~\ref{fig:W0301}(b)).

\item By lifting the vertical segment of $B_0$ over the attached 1-handle, the portion of the 0-handle boundary that is not occupied by any horizontal segments of $C_1$ or $B_0$ is pushed to the right. Consequently, the necessary endpoints (marked in purple) for the next column's operation will lie on the newly formed boundary.
\end{itemize}

The general form of the operation $\Theta(j)$ performed in an arbitrary column is illustrated in Figure~\ref{fig:W0302}.
\begin{itemize}
\item (p) depicts a gray region (a part of the surface consisting of the 0-handle, isotoped into a comb-like shape) and a vertical segment of the blue guide line $B_0$ intersected by horizontal segments (there may be any number of horizontal segments, including zero). To attach a 1-handle to the boundary of the surface, the two endpoints of the vertical segment (marked in purple) must lie on the boundary.
\item In  (q), we attach a 1-handle. A newly added red simple closed curve $C_j$ also appears. Here, as shown in  (s), the 1-handle is a part of a Hopf band with a full-twist of $-1$. Consequently, as illustrated in (t), the two bands are linked with a linking number of $-1$. 

By lifting this vertical segment over the 1-handle, the portion of the surface boundary where no horizontal segments of $C_i$ ($0 \leq i \leq j-1$) or $B_0$ exist is pushed to the right.
\end{itemize}

The surface obtained after performing the $j$-th operation $\Theta(j)$ is denoted by $SF(j)$.
\end{itemize}

This operation is repeated inductively for columns 2 through $N-1$. Let $SF$ denote the configuration obtained by removing the guide line $B_0$ from the final diagram $SF(N-1)$. (We will show later that $SF$ determines a PALF whose total space is diffeomorphic to $D^4$.)

\begin{itemize}
\item[\textbf{Step 4}] On the surface $SF$, we place the closed curve $C_0$ at the exact position from which the guide line $B_0$ was removed. We denote the resulting configuration by $P$.
\end{itemize}

\begin{figure}[htbp]
\centering
\begin{tikzpicture}
    \node[anchor=south west, inner sep=0] (image) at (0,0)  {\includegraphics[scale=0.75]{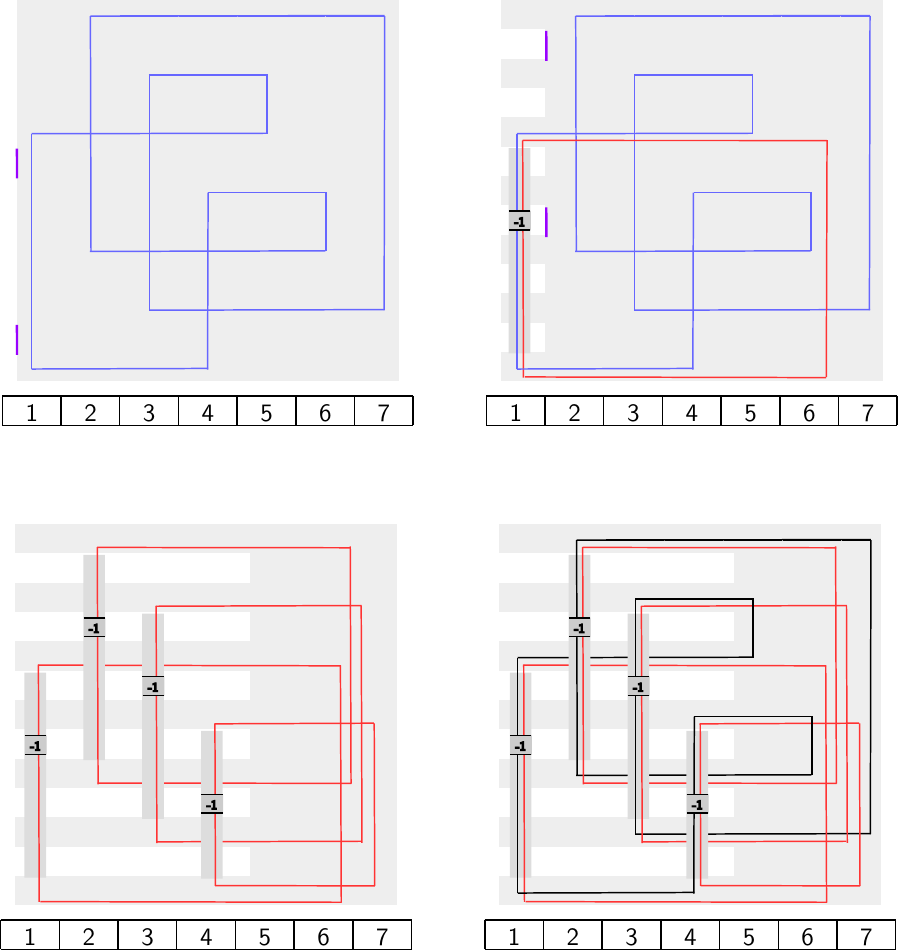}};
    \begin{scope}[x={(image.south east)},y={(image.north west)}]
    
        \node at (0.0, 0.705) {\small $B_0$};
        \node at (0.55, 0.705) {\small $B_0$};
    
        \node at (0.075, 0.27) {\small $C_1$};
        \node at (0.14, 0.40) {\small $C_2$};
        \node at (0.205, 0.335) {\small $C_3$};
        \node at (0.27, 0.212) {\small $C_4$};
        
        \node at (0.55, 0.27) {\small $C_0$};
        \node at (0.615, 0.27) {\small $C_1$};
        \node at (0.68, 0.40) {\small $C_2$};
        \node at (0.745, 0.335) {\small $C_3$};
        \node at (0.81, 0.212) {\small $C_4$};
        
        \node at (0.612, 0.825) {\small $C_1$};
        
        \node at (0.23, 0.51) {(a)};
        \node at (0.77, 0.51) {(b)};
        \node [below] at (0.23, -0.02) {(c)};
        \node [below] at (0.77, -0.02) {(d)};
    \end{scope}
\end{tikzpicture}
\caption{Construction of a PALF. (a) The initial surface $SF(0)$. (b) The surface $SF(1)$ obtained after the first operation. (c) The PALF $SF$ with monodromy factorization $(C_4, C_3, C_2, C_1)$. (d) The PALF $P$ with monodromy factorization $(C_0, C_4, C_3, C_2, C_1)$.}
\label{fig:W0301}
\end{figure}

\begin{figure}[htbp]
\centering  
\begin{tikzpicture}
    \node[anchor=south west, inner sep=0] (image) at (0,0)  {\includegraphics[scale=0.95]{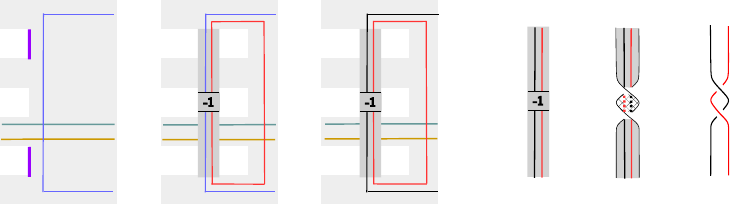} };
    \begin{scope}[x={(image.south east)},y={(image.north west)}]
        \node at (0.80, 0.5) {\huge $=$};
        \node [below] at (0.09, -0.02) {(p)};
        \node [below] at (0.31, -0.02) {(q)};
        \node [below] at (0.52, -0.02) {(r)};
        \node [below] at (0.79, -0.02) {(s)};
        \node [below] at (0.99, -0.02) {(t)};
    \end{scope}
\end{tikzpicture}
\caption{Attaching a 1-handle to lift a vertical segment.}
\label{fig:W0302}
\end{figure}

This completes the construction procedure. The constructed object $P$ determines a PALF because it satisfies the conditions given in Subsection~\ref{sec:preliminaries3}, which we verify as follows:

\begin{itemize}
\item The regular fiber is a surface with boundary, obtained by attaching several 1-handles to the boundary of a 0-handle.

\item By lifting the vertical segments in columns $2$ through $N-1$ over the 1-handles, all self-intersections of the closed curve $C_0$ are eliminated. Consequently, all the closed curves $C_i$ ($0 \leq i \leq N-1$) are simple closed curves.

\item The boundary of the resulting regular fiber is non-empty.

\item Each simple closed curve $C_i$ ($0 \leq i \leq N-1$) passes over at least one 1-handle and intersects the cocore of that 1-handle transversely at exactly one point. Therefore, each $C_i$ represents a nontrivial first homology class in the regular fiber.
\end{itemize}

Similarly, the configuration $SF$ also satisfies the conditions for forming a PALF. 
We assign the monodromy factorization $(C_{N-1}, \ldots, C_2, C_1)$ to the PALF $SF$, and the monodromy factorization $(C_0, C_{N-1}, \ldots, C_2, C_1)$ to the PALF $P$.

We now show that the total space of the PALF $SF$ is diffeomorphic to $D^4$. Recall from Subsection~\ref{sec:preliminaries3} the procedure $\Phi$, which constructs a Kirby diagram representing a 4-dimensional 2-handlebody diffeomorphic to the total space of a given PALF. For each $1 \leq i \leq N-1$, the vanishing cycle $C_i$ of $SF$ and the 1-handle it passes over are depicted in Figure~\ref{fig:W0302}(q) (note that the blue guide line $B_0$ is omitted). Applying the procedure $\Phi$, let $\overline{C_i}$ denote the attaching circle corresponding to the vanishing cycle $C_i$. Because $\overline{C_i}$ intersects the cocore of the 1-handle transversely at exactly one point, the 2-handle attached along $\overline{C_i}$ and the corresponding 1-handle form a 1-handle/2-handle canceling pair. 
Since all 1-handles and 2-handles cancel out in this manner, leaving only the 0-handle, the total space of $SF$ is indeed diffeomorphic to $D^4$.

Returning to our running example, we obtain the PALF $SF$ with the ordered collection of vanishing cycles $(C_1, C_2, C_3, C_4)$ and the PALF $P$ with the collection $(C_0, C_1, C_2, C_3, C_4)$, as shown in Figures~\ref{fig:W0301}(c) and (d), respectively. In this specific example, the regular fiber has genus 1 and three boundary components.\\

We now show that the following lemma holds, which justifies our construction.

\begin{lem}\label{lem:construction1}
The total space of the PALF obtained by the above construction, which corresponds to the ordered collection of simple closed curves on the surface with boundary, is diffeomorphic to the 4-dimensional handlebody consisting of a 0-handle and a single 2-handle attached along $\widetilde{C_0}$ with framing $tb(\widetilde{C_0})-1$.
\end{lem}

\begin{proof}
Throughout this paper, we adopt the following notation for a knot $C$ in grid position: let $tb(C)$ denote the Thurston--Bennequin number, $wr(C)$ the writhe, $fr(C)$ the framing, and $sf(C)$ the surface framing of $C$. Furthermore, let $nw(C)$ (resp. $ne(C)$) denote the number of vertical segments of $C$ that have a northwest (NW) corner (resp. a northeast (NE) corner) as their upper endpoint and cross horizontal segments.

Recall our setup: the Stein surface $\Pi$ is diffeomorphic to the handlebody consisting of a 0-handle and a single 2-handle attached along $\widetilde{C_0}$ with framing $tb(\widetilde{C_0})-1$. 

The knot $\widetilde{C'_0}$ in grid position is obtained by converting the Legendrian knot $\widetilde{C_0}$ using the procedure described in Subsection~\ref{sec:preliminaries2}. By definition, we have:
\[
fr(\widetilde{C'_0}) = tb(\widetilde{C'_0}) - 1 = wr(\widetilde{C'_0}) - nw(\widetilde{C'_0}) - 1.
\]

The knot $\widetilde{C''_0}$ in grid position is obtained by applying SW stabilizations at NE corners to the vertical segments of $\widetilde{C'_0}$ that have NE corners and cross horizontal segments, followed by horizontal commutations.

Because this operation preserves the Legendrian knot type, $\widetilde{C''_0}$ is Legendrian isotopic to $\widetilde{C'_0}$. During this stabilization process, the writhe and the number of NW corners change as follows:
\begin{gather*} 
fr(\widetilde{C''_0}) = tb(\widetilde{C''_0}) - 1 = wr(\widetilde{C''_0}) - nw(\widetilde{C''_0}) - 1, \\
wr(\widetilde{C''_0}) = wr(\widetilde{C'_0}) + ne(\widetilde{C'_0}), \quad nw(\widetilde{C''_0}) = nw(\widetilde{C'_0}) + ne(\widetilde{C'_0}).
\end{gather*}
Substituting these relations, it immediately follows that:
\[
fr(\widetilde{C''_0}) = fr(\widetilde{C'_0}).
\]

Recall from Subsection~\ref{sec:preliminaries3} the procedure $\Phi$ for constructing a Kirby diagram $KD$ representing a 4-dimensional 2-handlebody diffeomorphic to the total space of a given PALF $P$. The regular fiber of the PALF $P$ contains $nw(\widetilde{C''_0})$ 1-handles, each with a full-twist of $-1$. Therefore, for the attaching circle $\overline{C_0}$ in $KD$, the following sequence of equalities holds for its framing:
\[fr(\overline{C_0}) = sf(\widetilde{C''_0}) - 1 = wr(\widetilde{C''_0}) - nw(\widetilde{C''_0}) - 1 = fr(\widetilde{C''_0}) = fr(\widetilde{C'_0}).\]

Let $j$ be an integer such that $1 \leq j \leq N-1$. We introduce the following condition, denoted by $\mathrm{ST}(j)$:

\textbf{Condition $\mathrm{ST}(j)$:}
The Kirby diagram $KD(j)$ contains the attaching circle $\overline{C_0}(j)$ and the attaching circles $\overline{C_i}$ for $1 \leq i \leq j$. Each attaching circle $\overline{C_i}$ ($1 \leq i \leq j$) passes over exactly one 1-handle. In the region to the right of the $j$-th column (i.e., the collection of columns with indices strictly greater than $j$), there are no 1-handles. Furthermore, in this region, the vertical stacking order of the attaching circles from bottom to top is $\overline{C_0}(j), \overline{C_j}, \ldots, \overline{C_1}$. 
The attaching circle $\overline{C_0}(j)$ has a framing equal to $fr(\widetilde{C'_0})$, and at any of its self-intersections, the vertical segment crosses over the horizontal segment.\\

\begin{figure}[htbp]
\centering  
\begin{tikzpicture}
    \node[anchor=south west, inner sep=0] (image) at (0,0)  {\includegraphics[scale=0.8]{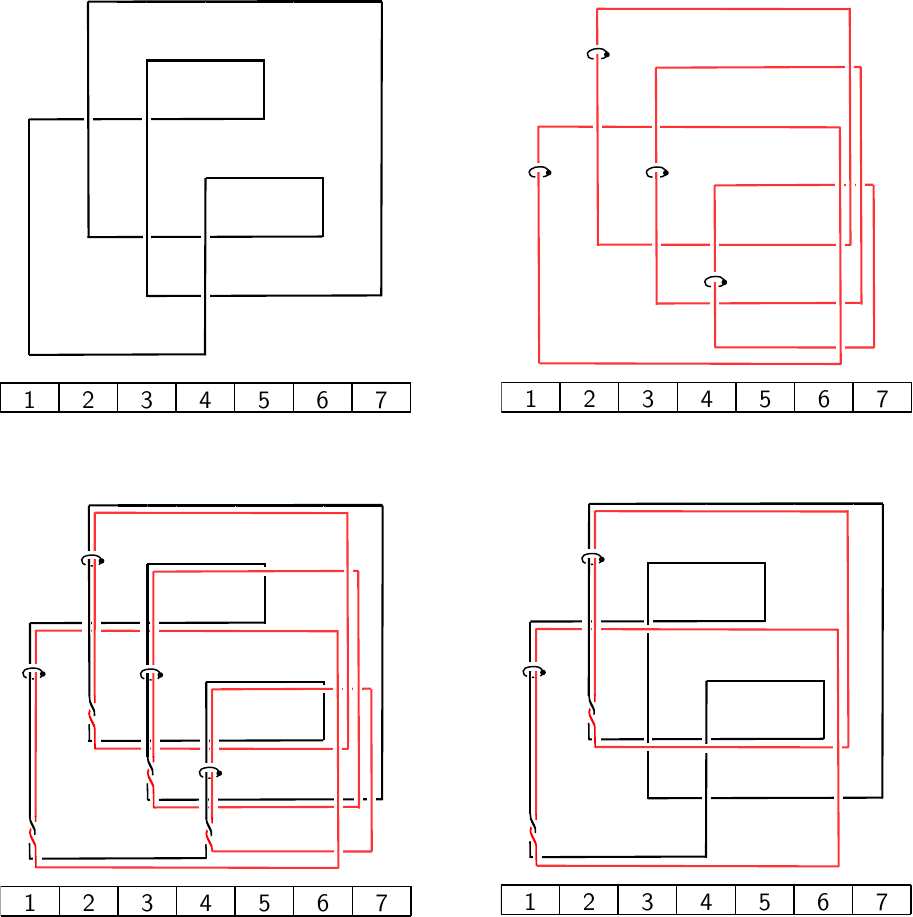}};
    \begin{scope}[x={(image.south east)},y={(image.north west)}]
    
        \node at (0.01, 0.835) {\small $\overline{C_0''}$};
    
        \node at (0.617, 0.835) {\small $\overline{C_1}$};
        \node at (0.683, 0.96) {\small $\overline{C_2}$};
        \node at (0.748, 0.895) {\small $\overline{C_3}$};
        \node at (0.813, 0.77) {\small $\overline{C_4}$};
    
        \node at (0.01, 0.29) {\small $\overline{C_0}$};
        \node at (0.065, 0.29) {\small $\overline{C_1}$};
        \node at (0.133, 0.41) {\small $\overline{C_2}$};
        \node at (0.198, 0.35) {\small $\overline{C_3}$};
        \node at (0.263, 0.225) {\small $\overline{C_4}$};
        
        \node at (0.54, 0.29) {\small $\overline{C_0}(2)$};
        \node at (0.612, 0.29) {\small $\overline{C_1}$};
        \node at (0.678, 0.41) {\small $\overline{C_2}$};
        
        \node [below] at (0.23, 0.53) {(a)};
        \node [below] at (0.77, 0.53) {(b)};
        \node [below] at (0.23, -0.02) {(c)};
        \node [below] at (0.77, -0.02) {(d)};
    \end{scope}
\end{tikzpicture}
\caption{(a) The Kirby diagram $KD''$. The framing of $\overline{C''_0}$ is $0$. (b) The Kirby diagram corresponding to the PALF $SF$. The framing of each $\overline{C_i}$ for $1 \leq i \leq 4$ is $-2$. (c) The Kirby diagram $KD$ corresponding to the PALF $P$. The framing of $\overline{C_0}$ is $0$, and the framing of each $\overline{C_i}$ for $1 \leq i \leq 4$ is $-2$. (d) The Kirby diagram $KD(2)$. The framing of $\overline{C_0}(2)$ is $0$, and the framing of each $\overline{C_i}$ for $i = 1, 2$ is $-2$. The Kirby diagrams $KD''$, $KD$, and $KD(2)$ represent the same diffeomorphism type.}
\label{fig:W0303}
\end{figure}

\begin{figure}[htbp]
\centering  
\begin{tikzpicture}
    \node[anchor=south west, inner sep=0] (image) at (0,0)  {\includegraphics[scale=0.55]{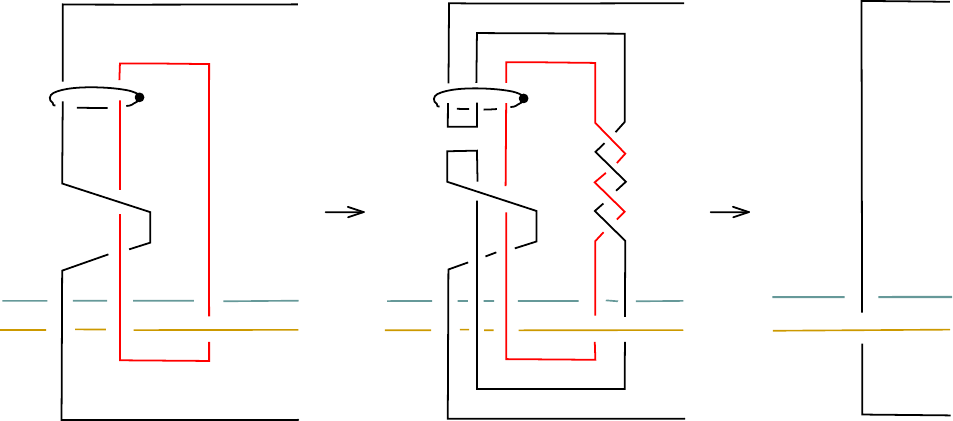}};
    \begin{scope}[x={(image.south east)},y={(image.north west)}]
        \node [below] at (0.19, -0.02) {(a)};
        \node [below] at (0.58, -0.02) {(b)};
        \node [below] at (0.93, -0.02) {(c)};
    \end{scope}
\end{tikzpicture}
\caption{The Kirby move corresponding to the operation $\Psi$.}
\label{fig:W0305}
\end{figure}

We set $KD(N-1) \coloneqq KD$ and $\overline{C_0}(N-1) \coloneqq \overline{C_0}$. By the construction of $KD$ from the PALF $P$, Condition $\mathrm{ST}(N-1)$ holds.

Assume that Condition $\mathrm{ST}(j)$ holds for some $j$ ($2 \leq j \leq N-1$). Under this assumption, we perform a sequence of Kirby moves, denoted by Operation $\Psi$.

By Condition $\mathrm{ST}(j)$, the attaching circle $\overline{C_j}$ can be deformed by isotopy into a neighborhood of the $j$-th column, as illustrated in Figure~\ref{fig:W0305}(a). In this figure, the red curve represents $\overline{C_j}$, and the black curve represents $\overline{C_0}(j)$.
Let the orange segments represent the portions of the attaching circles $\overline{C_i}$ ($1 \leq i \leq j-1$) that pass over the right part of $\overline{C_j}$ (i.e., the portion of the circle not lifted over the 1-handle). There may be any number of such orange segments, including zero. Similarly, let the dark green segments represent the portions of the attaching circle $\overline{C_0}(j)$ itself that pass under the depicted segment of $\overline{C_0}(j)$. There may also be any number of such dark green segments, including zero.

As shown in Figure~\ref{fig:W0305}, Operation $\Psi$ proceeds as follows: we slide the 2-handle $\overline{C_0}(j)$ over $\overline{C_j}$, and subsequently cancel the 1-handle/2-handle pair consisting of the dotted circle and $\overline{C_j}$.

Through this sequence of moves from (a) to (b) and finally to (c), we obtain the new attaching circle, denoted by $\overline{C_0}(j-1)$.

Following Operation $\Psi$, the orange segments (representing parts of $\overline{C_i}$ for $1 \leq i \leq j-1$) now pass over $\overline{C_0}(j-1)$. The dark green segments (representing parts of the original $\overline{C_0}$) pass under $\overline{C_0}(j-1)$. During this process, the framing remains invariant. Thus, the following equality holds:
$$fr(\overline{C_0}(j-1)) = fr(\widetilde{C'_0}).$$
As a result of Operation $\Psi$, Condition $\mathrm{ST}(j-1)$ holds, providing the premise for the operation in the subsequent $(j-1)$-th column.

By repeating this inductive procedure from $j=N-1$ down to $j=2$, we arrive at the Kirby diagram $KD(1)$ which satisfies Condition $\mathrm{ST}(1)$. Applying Operation $\Psi$ one final time to $KD(1)$ yields the Kirby diagram $KD(0)$. This final diagram $KD(0)$ contains no 1-handles and consists solely of the attaching circle $\overline{C_0}(0)$. Furthermore, in $KD(0)$, all vertical segments cross over the horizontal segments (consistent with the grid diagram convention), and the framing is $fr(\overline{C_0}(0)) = fr(\widetilde{C'_0})$. 

The 4-manifold represented by this final Kirby diagram $KD(0)$ is diffeomorphic to the 4-manifold represented by $KD''$. Since $KD''$ represents $\Pi$, we conclude that the total space of the PALF $P$ is diffeomorphic to the original Stein surface $\Pi$.

In our running concrete example, the Kirby diagram $KD$ in Figure~\ref{fig:W0303}(c) corresponds to the PALF $P$. Figure~\ref{fig:W0303}(d) shows the Kirby diagram $KD(2)$, which satisfies Condition $\mathrm{ST}(2)$. The Kirby diagrams $KD''$, $KD$, and $KD(2)$ represent the same diffeomorphism type. Furthermore, the attaching circles $\overline{C''_0}$, $\overline{C_0}$, and $\overline{C_0}(2)$ all have the same framing of $0$.
\end{proof}

\subsection{The general case of a 2-handlebody with 1-handles}\label{subsec:construction2}

We now describe a method for constructing a PALF in the general case, where the handle decomposition of the given 4-manifold consists of a 0-handle, 1-handles, and 2-handles. Step 0 is based on the method introduced by Akbulut and Ozbagci \cite{MR1825664} (see also \cite[Section 10.2]{MR2114165}).

\begin{itemize}
\item[\textbf{Step 0}] We assume that each 1-handle is represented by a pair of 3-balls over the front projection of a Legendrian tangle, and that the framed link diagram is initially in standard form (see Figure~\ref{fig:W0202}). We modify this handle decomposition by applying a single negative full-twist to the strands passing through each dotted circle (see Figure~\ref{fig:W0306}(a)). 

This twisting operation preserves the diffeomorphism type of the underlying 4-manifold.
\end{itemize}

We denote this modified 2-handlebody by $\Pi$, identifying it with the original Stein surface. We assume that $\Pi$ has a handle decomposition consisting of a single 0-handle, $\ell$ 1-handles, and $m$ 2-handles. These 2-handles are attached along Legendrian knots $\widetilde{C_{01}}, \ldots, \widetilde{C_{0m}}$ in the boundary of the union of the 0- and 1-handles, with a framing of $tb(\widetilde{C_{0k}}) - 1$ for each $1 \leq k \leq m$.

\begin{figure}[htbp]
\centering  
\begin{tikzpicture}
    \node[anchor=south west, inner sep=0] (image) at (0,0)  {\includegraphics[scale=0.67]{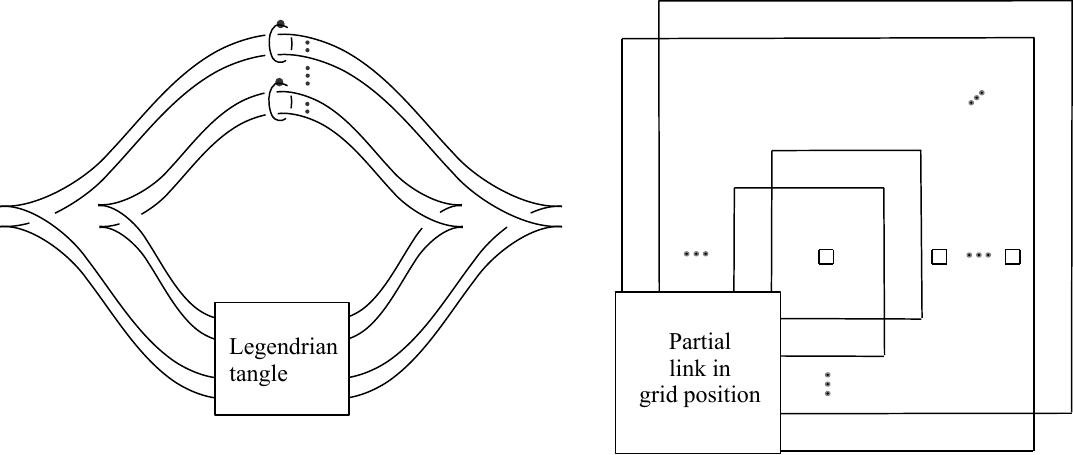}};
    \begin{scope}[x={(image.south east)},y={(image.north west)}]
        \node [below] at (0.27, -0.02) {(a)};
        \node [below] at (0.76, -0.02) {(b)};
    \end{scope}
\end{tikzpicture}
\caption{(a) A front projection of a Legendrian link. (b) The corresponding link in grid position.}
\label{fig:W0306}
\end{figure}

\begin{itemize}
\item[\textbf{Step 1}] Following the procedure described in Subsection~\ref{sec:preliminaries2}, we convert the components $\widetilde{C_{0k}}$ of the Legendrian link into components $\widetilde{C'_{0k}}$ forming a link in grid position, as shown in Figure~\ref{fig:W0306}(b). Here, the $\ell$ small squares represent the $\ell$ holes of $\natural^\ell (S^1 \times D^1)$, and the arcs of the link cannot be moved across these small squares by isotopy. 

Next, to ensure that all vertical segments of the link in grid position to be lifted over 1-handles have northwest (NW) corners as their upper endpoints, we apply SW stabilizations at NE corners and subsequent horizontal commutations to any vertical segments whose upper endpoints are northeast (NE) corners. Through these operations, we convert the link $\widetilde{C'_{0k}}$ in grid position into a new link $\widetilde{C''_{0k}}$ in grid position.
\end{itemize}

Let $KD''$ denote the Kirby diagram consisting of the 2-handles attached to $\natural^\ell (S^1 \times D^3)$ along the attaching circles $\overline{C''_{0k}}$ determined by $\widetilde{C''_{0k}}$. The 4-manifold represented by $KD''$ is diffeomorphic to the Stein surface $\Pi$ (see Figure~\ref{fig:W0307}(a)).  

\begin{itemize}
\item[\textbf{Step 2}] Let $SF(0)$ denote the initial surface with boundary $\natural^\ell (S^1 \times D^1)$, which is represented as a gray square with $\ell$ holes. To guide the subsequent operations, we draw a copy of the component $\widetilde{C''_{0k}}$ in grid position on $SF(0)$ for each $1 \leq k \leq m$. We denote these copies by $B_{0k}$ (see Figure~\ref{fig:W0307}(b)).  
\end{itemize}

\begin{figure}[htbp]
\centering  
\begin{tikzpicture}
    \node[anchor=south west, inner sep=0] (image) at (0,0)  {\includegraphics[scale=0.67]{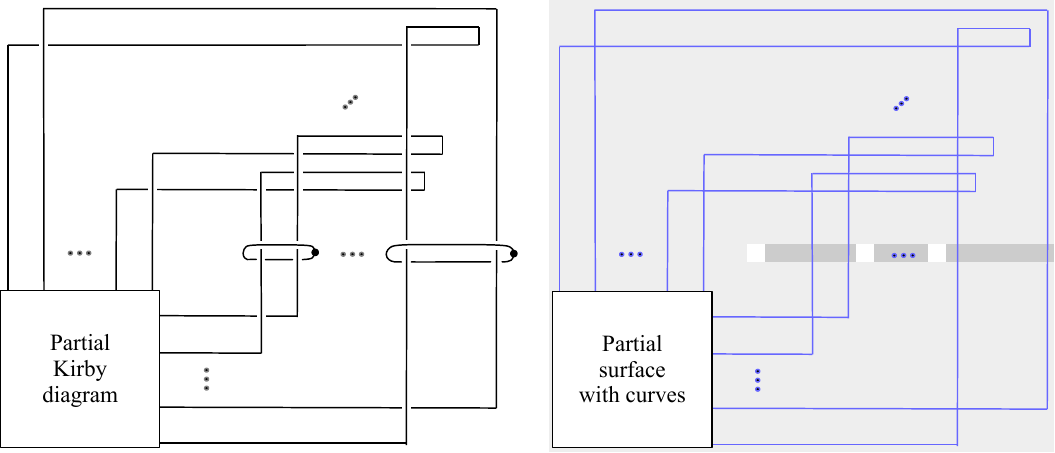}};
    \begin{scope}[x={(image.south east)},y={(image.north west)}]
        \node [below] at (0.27, -0.02) {(a)};
        \node [below] at (0.76, -0.02) {(b)};
    \end{scope}
\end{tikzpicture}
\caption{(a) The Kirby diagram $KD''$. (b) The surface $\natural^\ell (S^1 \times D^1)$ with $\ell$ holes and the blue guide lines $B_{0k}$.}
\label{fig:W0307}
\end{figure}

\begin{itemize}
\item[\textbf{Step 3}] We perform a sequence of operations, denoted by $\Theta(j)$ for $1 \leq j \leq N-1$, starting on $SF(0)$. The first operation, $\Theta(1)$, proceeds exactly as in Subsection~\ref{subsec:construction1}, with the additional constraint that the newly added simple closed curve $C_1$ does not pass across any of the $\ell$ holes of $\natural^\ell (S^1 \times D^1)$. 
The surface obtained after performing the $j$-th operation $\Theta(j)$ is denoted by $SF(j)$. 
\end{itemize}

This operation is repeated inductively for columns 2 through $N-1$. Let $SF$ denote the configuration obtained by removing the guide lines $B_{0k}$ ($1 \leq k \leq m$) from the final diagram $SF(N-1)$. (We will show later that $SF$ determines a PALF whose total space is diffeomorphic to $\natural^\ell (S^1 \times D^3)$.)

\begin{itemize}
\item[\textbf{Step 4}] On the surface $SF$, we place the closed curves $C_{0k}$ at the exact positions from which the guide lines $B_{0k}$ were removed. We denote the resulting configuration by $P$.
\end{itemize}

This completes the construction procedure. The constructed configuration $P$ satisfies the conditions for forming a PALF given in Subsection~\ref{sec:preliminaries3}, in the same way as demonstrated in Subsection~\ref{subsec:construction1}. 

Similarly, the configuration $SF$ also satisfies the conditions for forming a PALF. 
We assign the monodromy factorization $(C_{N-1}, \ldots, C_2, C_1)$ to the PALF $SF$, and the monodromy factorization $(C_{0m}, \ldots, C_{01}, C_{N-1}, \ldots, C_2, C_1)$ to the PALF $P$. 

We now show that the total space of the PALF $SF$ is diffeomorphic to $\natural^\ell (S^1 \times D^3)$.
Recall from Subsection~\ref{sec:preliminaries3} the procedure $\Phi$, which constructs a Kirby diagram representing a 4-dimensional 2-handlebody diffeomorphic to the total space of a given PALF. For each $1 \leq i \leq N-1$, the vanishing cycle $C_i$ of $SF$ and the 1-handle it passes over are depicted in Figure~\ref{fig:W0302}(q)(note that the blue guide line $B_{0k}$ is omitted). Applying the procedure $\Phi$, let $\overline{C_i}$ denote the attaching circle corresponding to the vanishing cycle $C_i$. Because $\overline{C_i}$ intersects the cocore of the 1-handle transversely at exactly one point, the 2-handle attached along $\overline{C_i}$ and the corresponding 1-handle form a 1-handle/2-handle canceling pair. 
Since all such 1-handles and 2-handles cancel out in this manner, leaving only the 0-handle and the initial $\ell$ 1-handles, the total space of $SF$ is indeed diffeomorphic to $\natural^\ell (S^1 \times D^3)$.

\begin{figure}[htbp]
\centering  
\begin{tikzpicture}
    \node[anchor=south west, inner sep=0] (image) at (0,0)  {\includegraphics[scale=0.65]{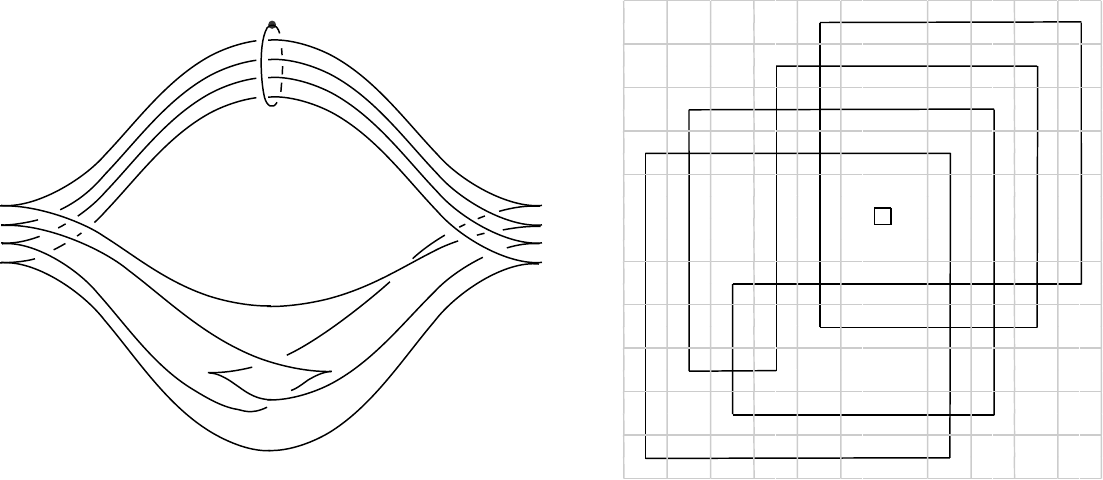}};
    \begin{scope}[x={(image.south east)},y={(image.north west)}]
        \node at (0.04, 0.67) {\small $\widetilde{C_{02}}$};
        \node at (0.04, 0.35) {\small $\widetilde{C_{01}}$};
        \node at (0.56, 0.5) {\small $\widetilde{C'_{01}}$};
        \node at (0.6, 0.76) {\small $\widetilde{C'_{02}}$};
        \node [below] at (0.27, -0.02) {(a)};
        \node [below] at (0.76, -0.02) {(b)};
    \end{scope}
\end{tikzpicture}
\caption{(a) The Legendrian link diagram for a concrete example with a 1-handle. (b) The corresponding link in grid position.}
\label{fig:W0308}
\end{figure}

\begin{figure}[htbp]
\centering  
\begin{tikzpicture}
    \node[anchor=south west, inner sep=0] (image) at (0,0)  {\includegraphics[scale=0.8]{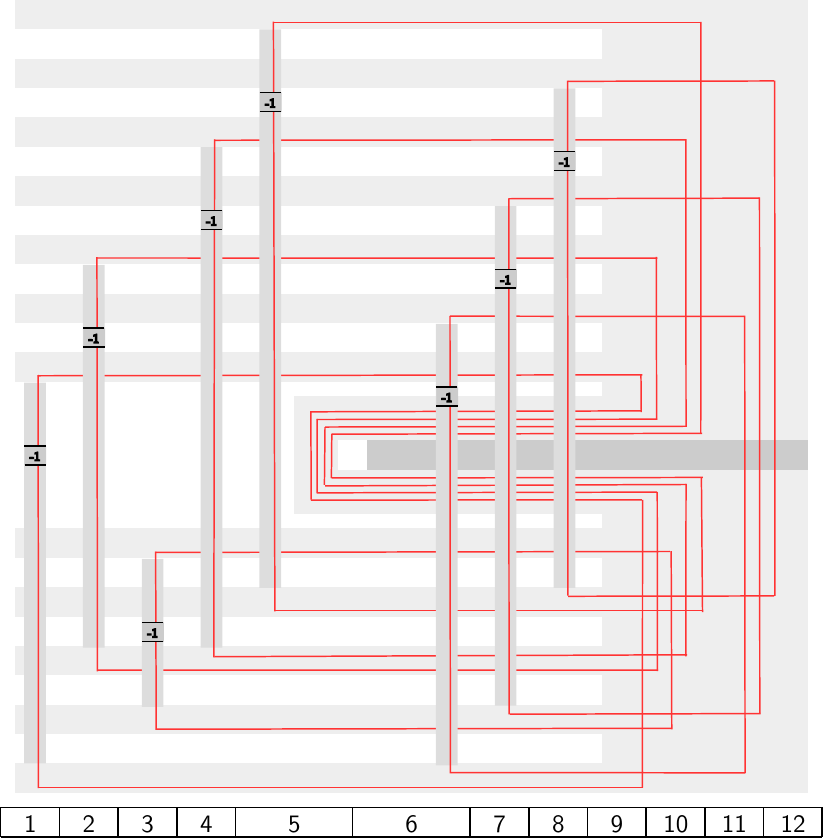}};
    \begin{scope}[x={(image.south east)},y={(image.north west)}]
        \node at (0.085, 0.52) {\small $C_1$};
        \node at (0.155, 0.665) {\small $C_2$};
        \node at (0.225, 0.315) {\small $C_3$};
        \node at (0.295, 0.806) {\small $C_4$};
        \node at (0.368, 0.94) {\small $C_5$};
        \node at (0.58, 0.59) {\small $C_6$};
        \node at (0.65, 0.735) {\small $C_7$};
        \node at (0.723, 0.875) {\small $C_8$};
    \end{scope}
\end{tikzpicture}
\caption{The PALF $SF$ with monodromy factorization $(C_8, C_7, C_6, C_5, C_4, C_3, C_2, C_1)$.}
\label{fig:W0309}
\end{figure}

\begin{figure}[htbp]
\centering  
\begin{tikzpicture}
    \node[anchor=south west, inner sep=0] (image) at (0,0)  {\includegraphics[scale=0.8]{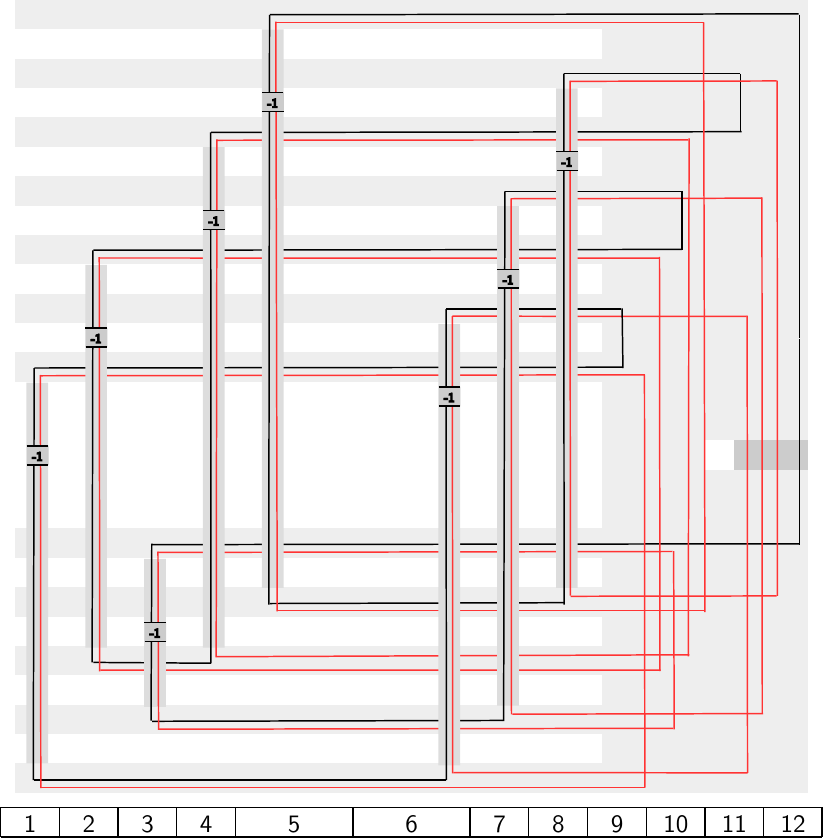}};
    \begin{scope}[x={(image.south east)},y={(image.north west)}]
        \node at (0.012, 0.52) {\small $C_{01}$};
        \node at (0.082, 0.665) {\small $C_{02}$};
        \node at (0.085, 0.52) {\small $C_1$};
        \node at (0.155, 0.665) {\small $C_2$};
        \node at (0.225, 0.315) {\small $C_3$};
        \node at (0.295, 0.806) {\small $C_4$};
        \node at (0.368, 0.94) {\small $C_5$};
        \node at (0.58, 0.59) {\small $C_6$};
        \node at (0.65, 0.735) {\small $C_7$};
        \node at (0.723, 0.875) {\small $C_8$};
    \end{scope}
\end{tikzpicture}
\caption{The PALF $P$ with monodromy factorization $(C_{02}, C_{01}, C_8, C_7, C_6, C_5, C_4, C_3, C_2, C_1)$.}
\label{fig:W0310}
\end{figure}

We now present a concrete example. We consider the original Stein surface $\Pi$ consisting of a 0-handle, one 1-handle, and two 2-handles. These 2-handles are attached along the components $\widetilde{C_{01}}$ and $\widetilde{C_{02}}$ of a Legendrian link, corresponding to the configuration in Step 0, as shown in Figure~\ref{fig:W0308}(a). (The framings of $\widetilde{C_{01}}$ and $\widetilde{C_{02}}$ are $-2$ and $0$, respectively.) 

Following the procedure described in Subsection~\ref{sec:preliminaries2}, we convert the Legendrian link components $\widetilde{C_{01}}$ and $\widetilde{C_{02}}$ into link components $\widetilde{C'_{01}}$ and $\widetilde{C'_{02}}$ in grid position, as shown in Figure~\ref{fig:W0308}(b). Here, the single small square represents the hole of $S^1 \times D^1$, and the arcs of the link cannot be moved across this small square by isotopy. This depicts a simplified version of a link diagram associated with a cork.

Applying our procedure, we construct the PALF $SF$ with the ordered collection of vanishing cycles $(C_8, \ldots, C_1)$ (see Figure~\ref{fig:W0309}). 
After constructing the PALF $SF$, we can further perform the following modification: via isotopy, we move the hole of $S^1 \times D^1$ and the rightmost portions of four vanishing cycles near column 12. On the surface $SF$, we then place the closed curves $C_{01}$ and $C_{02}$ at the exact positions from which the guide lines $B_{01}$ and $B_{02}$ were removed, thereby obtaining the PALF $P$ (see Figure~\ref{fig:W0310}). 

In this specific example, the regular fiber has genus 3 and four boundary components. (In the Kirby diagram associated with the PALF $P$, the framings of the attaching circles $\overline{C_{01}}$ and $\overline{C_{02}}$ are $-2$ and $0$, respectively.)

We now state and prove the main theorem of this section, which guarantees the correctness of our general construction.

\begin{thm}\label{thm:construction2}
The total space of the PALF obtained by the above construction, which corresponds to the ordered collection of simple closed curves on the surface with boundary, is diffeomorphic to the handlebody consisting of a 0-handle, $\ell$ 1-handles, and $m$ 2-handles attached along the Legendrian link components $\widetilde{C_{0k}}$ ($1 \leq k \leq m$) with framings $tb(\widetilde{C_{0k}})-1$.
\end{thm}

\begin{proof}
To prove this theorem, it suffices to consider the two main differences between this general setting and the setting of Lemma~\ref{lem:construction1}. The notation used below is identical to that in Lemma~\ref{lem:construction1}, except that the single component $C_0$ is replaced by the link components $C_{0k}$ for $1 \leq k \leq m$.

\begin{itemize}
\item[(1)] \textbf{The handle decomposition contains one or more 1-handles.} \\
In the link $\widetilde{C'_{0k}}$ in grid position shown in Figure~\ref{fig:W0306}(b), the arcs of the link cannot be moved by isotopy across the $\ell$ small squares, which represent the $\ell$ holes of $\natural^\ell (S^1 \times D^1)$. Similarly, the vanishing cycles $C_j$ ($1 \leq j$) generated during the PALF construction cannot be moved across these holes by isotopy. Nevertheless, the PALF can still be successfully constructed. Furthermore, after generating the PALF $SF$, we can perform the following modification: via isotopy, we move the $\ell$ holes of $\natural^\ell (S^1 \times D^1)$ and the rightmost portions of some vanishing cycles near column $N$.

\item[(2)] \textbf{The handle decomposition contains multiple 2-handles.} \\
First, we verify that the equality $fr(\overline{C_{0k}}) = fr(\widetilde{C'_{0k}})$ holds for each $k$. By definition, we have:
\begin{align*}
fr(\widetilde{C'_{0k}}) &= tb(\widetilde{C'_{0k}}) - 1 = wr(\widetilde{C'_{0k}}) - nw(\widetilde{C'_{0k}}) - 1, \\
fr(\widetilde{C''_{0k}}) &= tb(\widetilde{C''_{0k}}) - 1 = wr(\widetilde{C''_{0k}}) - nw(\widetilde{C''_{0k}}) - 1.
\end{align*}
During the stabilization process, the writhe and the number of NW corners change as follows: $wr(\widetilde{C''_{0k}}) = wr(\widetilde{C'_{0k}}) + ne(\widetilde{C'_{0k}})$ and $nw(\widetilde{C''_{0k}}) = nw(\widetilde{C'_{0k}}) + ne(\widetilde{C'_{0k}})$. Substituting these relations, it follows that:
\[
fr(\widetilde{C''_{0k}}) = fr(\widetilde{C'_{0k}}).
\]
Furthermore, for the attaching circles in the Kirby diagram, we have:
\begin{align*}
fr(\overline{C_{0k}}) &= sf(\widetilde{C''_{0k}}) - 1 = wr(\widetilde{C''_{0k}}) - nw(\widetilde{C''_{0k}}) - 1 \\
&= fr(\widetilde{C''_{0k}}) = fr(\widetilde{C'_{0k}}).
\end{align*}

As in the proof of Lemma~\ref{lem:construction1}, we perform the sequence of Kirby moves illustrated in Figure~\ref{fig:W0305}. In this figure, assume that the 1-handle is located in column $j$. The red curve represents $\overline{C_j}$, and the black curve represents $\overline{C_{0k}}(j)$. Let the orange segments represent the portions of the attaching circles $\overline{C_i}$ ($1 \leq i \leq j-1$) that pass over the right part of $\overline{C_j}$ (i.e., the portion of the circle not lifted over the 1-handle). There may be any number of such orange segments, including zero. Similarly, let the dark green segments represent portions of the components $\overline{C_{01}}(j), \ldots, \overline{C_{0m}}(j)$, including $\overline{C_{0k}}(j)$. There may also be any number of such dark green segments, including zero.

We slide the 2-handle $\overline{C_{0k}}(j)$ over $\overline{C_j}$, and subsequently cancel the 1-handle/2-handle pair consisting of the dotted circle and $\overline{C_j}$. 
By repeating this operation down to column $1$, we obtain a Kirby diagram $KD(0)$ whose attaching circles consist solely of $\overline{C_{0k}}(0)$ ($1 \leq k \leq m$). This diagram $KD(0)$ corresponds to a grid diagram where all vertical segments cross over the horizontal segments, with the framings given by $fr(\overline{C_{0k}}(0)) = fr(\widetilde{C'_{0k}})$.
The 4-manifold represented by this final Kirby diagram $KD(0)$ is diffeomorphic to the 4-manifold represented by $KD''$.
\end{itemize}

From the above arguments, we conclude that the total space of the PALF $P$ is diffeomorphic to the original Stein surface $\Pi$.
\end{proof}

\section{Minimal genus of a regular fiber of a PALF on a knot trace}\label{sec:genus}

In this section, by a knot (or link) trace, we mean a compact 4-manifold obtained by attaching 2-handles to a 0-handle $D^4$ along a framed knot (or link) in $S^3$, that is, a 2-handlebody without any 1-handles. 
Consider a knot trace that admits a Stein surface, whose attaching circle is a knot $K$ framed by $\overline{tb}(K) - 1$, where $\overline{tb}(K)$ is the maximal Thurston--Bennequin number. The minimal genus of a regular fiber of any PALF on such a knot trace defines an invariant of the knot $K$. We define:
\[
g(K) \coloneqq \min \left\{
\begin{array}{l}
\text{genus of a regular fiber of a PALF on the knot trace} \\
\text{that admits a Stein surface,} \\
\text{whose attaching circle is the knot } K \text{ framed by } \overline{tb}(K) - 1
\end{array}
\right\}
\]

This definition naturally extends to the case of a link trace.  
Let $\overline{K}$ denote the mirror of a knot (or link) $K$.  
In general, $g(K)$ and $g(\overline{K})$ are not equal.

The minimal size of a grid diagram representing a knot (or link) $K$ is called the grid number of $K$. Using this grid number $N$, we obtain the following theorem for general knots and links:

\begin{thm}\label{thm:genus}
Let $K$ be a knot (or link) with grid number $N$. Then the minimal genus $g(K)$ of a regular fiber of a PALF on the knot (or link) trace of $K$ satisfies $g(K) \leq (N - 1)/2$.
\end{thm}

\begin{proof}
Let $\Sigma^b_g$ be the regular fiber of genus $g$ with $b$ boundary components obtained from our construction.  
This surface is formed by attaching at most $N - 1$ $1$-handles to a $0$-handle.  
Since the number of boundary components $b$ is at least $1$, applying the Euler characteristic formula
\[
2 - 2g - b = 1 - (\text{number of } 1\text{-handles})
\]
implies that the genus $g$ satisfies $g \leq (N - 1)/2$.

\end{proof}

\begin{rem}\label{rem:torus}
Focusing on the flexibility of the regular fiber, we present a combinatorial extension of this construction method in \cite{Tanaka2026combinatorial}. This yields a PALF whose total space is diffeomorphic to the original Stein surface, but with a different regular fiber. 
Applications of the construction method presented in this paper are given in \cite{Tanaka2026Alternating}. 
In particular, it is shown that applying this method to knot traces whose attaching circles are positive torus knots yields PALFs with regular fibers of genus $1$.
\end{rem}


\bibliographystyle{amsalpha}
\bibliography{ALF_1}

\end{document}